\documentclass[12pt]{article}

\usepackage[all]{xypic}
\usepackage{amsmath}
\usepackage{amsfonts}
\usepackage{amssymb}
\usepackage{amsthm}
\usepackage{color}

\usepackage[top=2cm, bottom=2cm, left=2cm, right=2cm]{geometry}
\theoremstyle{plain}
\newtheorem{thm}{Theorem}[section]

\theoremstyle{plain}
\newtheorem{crl}[thm]{Corollary}

\newtheorem{prp}[thm]{Proposition}

\newtheorem{lmm}[thm]{Lemma}

\newtheorem{conj}[thm]{Conjecture}

\newtheorem{rmk}[thm]{Remark}

\newcommand {\tb}{\textbf}
\newcommand {\mb}{\mathbb}
\newcommand {\Z}{\mb Z}
\newcommand {\R}{\mb R}
\newcommand {\C}{\mb C}

\newcommand {\colim}{\textrm{colim}\ }

\newcommand {\lra}{\longrightarrow}

\begin{document}

\title{{Freudenthal theorem} and spherical classes in $H_*QS^0$}
%finite loop spaces of spheres and unstable bordism groups of immersions}
 %and stable bordism of unstable bordism elements}

\author{Hadi Zare\\
        School of Mathematics, Statistics,
        and Computer Sciences\\
        University of Tehran, Tehran, Iran\\
        \textit{email:hadi.zare} at \textit{ut.ac.ir}}
\date{}

\maketitle

\begin{abstract}
This note is on spherical classes in $H_*(QS^0;k)$ when $k=\Z,\Z/p$ with a special focus on the case of $p=2$ related to Curtis conjecture.
% and $H_*(QX;\Z/2)$ when $X$ is path connected, respectively.
%We obtain results in theoretically three different levels of strength.
%First, applying Freudenthal's theorem obtain the result that the image of the integral unstable Hurewivz homomorphism ${\pi_*^s}\simeq\pi_*QS^0\to H_*(QS^0;\Z/2)$ when restricted to decomposable elements is given by $\Z\{h(\eta^2),h(\nu^2),h(\sigma^2)\}$.
{We apply Freudenthal theorem to prove a vanishing result for the Hurewicz image of elements in} ${\pi_*^s}$ that factor through certain finite spectra. Either in $p$-local or $p$-complete settings, this immediately implies that elements of well known infinite families in ${_p\pi_*^s}$, such as Mahowaldean families, map trivially under the unstable Hurewicz homomorphism ${_p\pi_*^s}\simeq{_p\pi_*}QS^0\to H_*(QS^0;\Z/p)$. We also observe that the image of the integral unstable Hurewicz homomorphism $\pi_*^s\simeq\pi_*QS^0\to H_*(QS^0;\Z)$ when restricted to the submodule of decomposable elements, is given by $\Z\{h(\eta^2),h(\nu^2),h(\sigma^2)\}$. We apply this latter to completely determine spherical classes in $H_*(\Omega^dS^{n+d};\Z/2)$ for certain values of $n>0$ and $d>0$; this verifies a Eccles' conjecture on spherical classes in $H_*QS^n$, $n>0$, on finite loop spaces associated to spheres.
\end{abstract}


\tableofcontents

\section{Introduction and statement of results}
%Notations.
%$$\begin{array}{lll}
%\Z/p     & \textrm{integers modulo} p,\\
%\Z_p     & \textrm{$p$-adic integers }=\lim\Z/p^i=\{\sum_{k=0}^{+\infty}a_kp^k:a_k\in\Z/p\},\\
%\Z_{(p)} & \textrm{integers localised at $p$ }=\{n/q:n\in\Z,q\neq p^i\}.
%\end{array}$$

%At a prime $p$, $p$-completion of a spectrum/space $E$ which we denote by $\widehat{E_p}$ is the same as $H\Z/p$-localisation of $E$ in the sense of Bousfield, that is $\widehat{E_p}=E_{H\Z/p}$. Consequently, $\widehat{E_p}$ captures all $\Z/p$-homology of $E$. Moreover, the infinite loop spaces are nilpotent. This implies that the completion and $\Omega^\infty$ functors commutes, i.e. $\widehat{(\Omega^\infty E)_p}=\Omega^\infty\widehat{E_p}$. For this reason, we work in the $p$-completed world and drop $\widehat{-}$ and $(-)_p$ from our notation, and simply write $E$ for its $p$-completion.

Let $QS^0=\colim \Omega^iS^i$ be the infinite loop space associated to the sphere spectrum. Curtis conjecture then reads as follows
(see \cite[Proposition 7.1]{Curtis} and \cite{Wellington} for more discussions).

\begin{conj}[{Curtis Conjecture}]\label{conjecture}
In positive degrees only the Hopf invariant one and Kervaire invariant one elements survive under the unstable Hurewicz homomorphism $h:{_2\pi_*^s}\simeq{_2\pi_*}QS^0\lra H_*QS^0$.
\end{conj}

Here and throughout, we write ${_p\pi_*^s}$ and ${_p\pi_*}$ for the $p$-components of $\pi_*^s$ and $\pi_*$, respectively. We also write $H_*$ for $H_*(-;k)$ where the coefficient ring will be clear from the context with an interest in $k=\Z/p$; although some of our results such as Theorem \ref{main1} holds when $k$ is an arbitrary $p$-local commutative coefficient ring. We work $p$-locally or $p$-complete which again will be clear from the context. \\

{We wish to examine Conjecture \ref{conjecture} in the light of Freudenthal theorem.} Suppose $E$ is a connected and connective spectrum, i.e. $\pi_iE\simeq0$ for $i<1$. We write $\Omega^\infty E=\colim \Omega^iE_i$ for the infinite loop space associated to $E$, $QX$ for $\Omega^\infty(\Sigma^\infty X)$ when $X$ is a space, and $\epsilon:\Sigma^\infty \Omega^\infty E\to E$ for the evaluation map which is the stable adjoint to the identity $\Omega^\infty E\to\Omega^\infty E$. The map $\epsilon$ induces stable homology suspension $H_*\Omega^\infty E\to H_*E$. {The following is a consequence of Freudenthal theorem.}

%The following is due to Kuhn.

%\begin{thm}\label{Kuhn-Kahn-Priddy}
%(\cite[Theorem 1.1(i)]{Kuhn-extended})
%(i) After localisation at the prime $2$, for any space $Y$ and any connected and connective spectrum $E$, there is an exact sequence of homotopy groups
%$$[\Sigma^\infty Y,\Sigma^\infty D_2(\Omega^\infty X)]\lra[\Sigma^\infty Y,\Sigma^\infty\Omega^\infty E]\stackrel{\epsilon_*}{\lra}[\Sigma^\infty Y,E]\lra 0.$$
%(ii) After localisation at an odd prime $p$, for any space $Y$ and any connected and connective spectrum $E$, there is an exact sequence of homotopy groups
%$$[\Sigma^\infty Y,\Sigma^\infty D_2(\Omega^\infty E)\vee D_p(\Omega^\infty E)]\lra[\Sigma^\infty Y,\Sigma^\infty\Omega^\infty E]\stackrel{\epsilon_*}{\lra}[\Sigma^\infty Y,E]\lra 0.$$
%Here, $D_r(-)=E\Sigma_r\ltimes_{\Sigma_r}(-)^{\wedge r}$ is the $r$-adic construction on the category of pointed spaces.
%\end{thm}

%An immediate corollary is the following.

\begin{prp}\label{equiv1}
Suppose $E$ is $r$-connected with $r>0$. Then, $\epsilon$ is a weak $(2r+1)$-equivalence.
\end{prp}

\begin{proof}
{First, note that for a composition $X\stackrel{f}{\to}Y\stackrel{g}{\to}Z$ of maps among pointed spaces, the homotopy fibre of the induced map $\mathrm{Fib}(gf)\to\mathrm{Fib}(g)$, up to weak homotopy equivalence, can be identified with $\mathrm{Fib}(f)$. Here, $\mathrm{Fib}(f)$ denotes the homotopy fibre of $f$. In particular, for $X=Z$, if $gf$ is homotopic to the identity then $\mathrm{Fib}(f)$ is weak homotopy equivalent to $\Omega\mathrm{Fib}(g)$.\\
Now, let $E$ be an $r$-connected spectrum. Then, $X=\Omega^\infty E$ is an $r$-connected space. Since, $X$ is an infinite loop space then the composition $X\stackrel{\iota}{\to}QX\stackrel{\Omega^\infty\epsilon}{\to}X$ equals to identity \cite{May-G} (see also \cite{BE1}). It follows that there is a weak homotopy equivalence $\mathrm{Fib}(\Omega^\infty\epsilon)\to\Omega\mathrm{Fib}(\iota)$. On the other hand, by Freudenthal theorem, the map $\iota:X\to QX$ induces an isomorphism $\pi_iX\to\pi_iQX\simeq\pi_i^sX$ for $i\leqslant 2r$. This means that $\iota$ is a $2r$-equivalence, that is $\pi_i\mathrm{Fib}(\iota)\simeq 0$ for $i\leqslant 2r$, consequently $\pi_i\mathrm{Fib}(\Omega^\infty\epsilon)\simeq 0$ for $i\leqslant 2r+1$. This means that $(\Omega^\infty\epsilon)_*:\pi_iQ\Omega^\infty E\to\pi_i\Omega^\infty E$ is an isomorphism $i\leqslant 2r+1$. Equivalently, $\epsilon_*:\pi_i\Sigma^\infty\Omega^\infty E\to\pi_iE$ is an isomorphism for $i\leqslant 2r+1$.}% Now, the claimed equivalence follows from Whitehead's theorem for $CW$-spectra.
\end{proof}

{Note that if $E$ is a $CW$-spectrum, then the weak homotopy equivalence would give a homotopy equivalence. To state our next result, we fix our terminology; for a spectrum $E$ we refer to the Hurewicz homomorphism $h^s:\pi_*E\to H_*E$ as the stable Hurewicz homomorphism whereas, in positive degrees, we refer to $h:\pi_*E\simeq\pi_*\Omega^\infty E\to H_*\Omega^\infty E$ as the unstable Hurewicz homomorphism. We can apply the above equivalence to obtain some information on the image of the unstable Hurewicz homomorphism.} We have the following.

%The proof of this proposition is an straightforward application of Theorem \ref{Kuhn-Kahn-Priddy} upon choosing $Y=S^n$. We may approximate $E$ with a $CW$-spectrum with its bottom cell in dimension $r+1$. For $r>0$, since $\pi_iE\simeq\pi_i\Omega^\infty E$ for $i>-1$ then if $E$ is $r$-connected so would be $\Omega^\infty E$ as a space and has its bottom cell in dimension $r+1$. Consequently, $D_p(\Omega^\infty E)$ has its bottom cell in dimension $p(r+1)$. Hence, the presence of the $D_2(\Omega^\infty E)$ factor in Kuhn's theorem, in either of the cases, which has its bottom cell in dimension $2(r+1)$ results in vanishing of homotopy groups in dimensions less than $2(r+1)$ which proves the claim when localised at an arbitrary prime $p$, hence the integral result. %Since this holds at any arbitrary prime, it then holds integrally.\\

\begin{thm}\label{main1}
Suppose $f:S^n\lra QS^0$ is given so that for some spectrum $E$ there is a factorisation $S^n\stackrel{f_E}{\lra}\Omega^\infty E\lra QS^0$. Then, the following statements hold.\\
(i) If $E$ is $r$-connected, $n\leqslant 2r+1$, and $h^s(f_E)=0$, then $h(f)=0$. {In particular, if $E$ is a $CW$-spectrum and $f_E$ maps trivially under $p$-local stable Hurewicz map $h^s_{(p)}:{_p\pi_*}E\to H_*(E;\Z/p)$ only for some prime $p$ then $f$ maps trivially under the $p$-local unstable Hurewicz map ${_p\pi_*}QS^0\to H_*QS^0$. A similar statement holds in the $p$-complete setting.}\\
(ii) If the above factorisation is induced by a factorisation of a map of spectra
$S^n\stackrel{\alpha}{\lra}E\stackrel{c}{\lra}S^0$ where $E$ is a finite $CW$-spectrum of dimension $r$, $n\geqslant 2r+1$, and $c_*=0$ then $h(f)=0$. {In particular, if $c_*=0$ holds $p$-locally only for some prime $p$ then $f$ maps trivially under the $p$-local unstable Hurewicz map ${_p\pi_*}QS^0\to H_*QS^0$. A similar statement holds in the $p$-complete setting.}
\end{thm}

%The statement of this result was proposed by an anonymous referee as a consequence of Freudenthal's theorem on which the first draft of this paper was based on. We like to acknowledge this contribution. However, we were not able to derive the theorem from Freudenthal's theorem and instead we obtain it as a simple consequence of Kuhn's theorem.
{As the first application of the $p$-local/$p$-complete version of the above Theorem, we have the following.}

\begin{thm}\label{infinitefamily}
At the prime $p=2$, let $f$ be an element of one of the following families (some of which are only known to be potentially existing):\\
(i) Mahowald's family, $\eta_i\in{_2\pi_{2^i}}QS^0$ with $i\geqslant 3$, detected by $h_1h_i$ in the ASS;\\
(ii) Lin's family in ${_2\pi_{2^{i}+18}^s}$, $i>10$, detected by $c_1h_i$ in the Adams spectral sequence;\\
(iii) Lin's family in ${_2\pi_{2^{i+1}+37}^s}$, $i>12$, detected by $h_i^2h_3d_1$ in the Adams spectral sequence;\\
(iv) Bruner's family, $\tau_i\in {_2\pi_{2^{i+1}+1}^s}$, $i\geqslant 5$, detected by $h_2h_i^2$ in the ASS;\\
Then, for $h:{_2\pi_*^s}\simeq{_2\pi_*}QS^0\to H_*QS^0$ we have $h(f)=0$.\\
At an odd prime $p$, the elements of\\
(vi) Cohen's family in ${_p\pi_{2(p-1)p^i+2p-5}^s}$ detected by $h_0b_{j-1}$\\
map trivially $h:{_p\pi_*^s}\simeq{_p\pi_*}QS^0\to H_*QS^0$.
\end{thm}

The families mentioned in the above theorem, are well known families that are known (some only potentially) to exist in ${_p\pi_*^s}$. There exist also hypothetical ways to construct infinite families in ${_p\pi_*^s}$ in the literature which are constructed using certain maps $S^{2^i}\to\R P^{2^{i-3}}$ or $S^{2^i}\to T({2^{i-3}})$ as well as some elements factoring through $\R P^2$ or $\C P^2$; examples of families constructed by using either one of these constructions can be found in \cite[Examples 4.5, 4.9, 5.4, 5.6]{Kuhn-construction} at $p=2$, and \cite[Example 3.9]{HunterKuhn} at odd primes. We have the following.

\begin{thm}\label{infinitefamily-2}
($p=2$) Let $f$ be an element of one of the families constructed in Kuhn that are detected by \\
(i) $bh_2h_i$; (ii) $bh_3h_{i-1}^2$; (iii) $ah_{i-1}^2$\\
where $a$ and $b$ are certain elements in the ASS as in \cite[Theorems 4.4, 5.3]{Kuhn-construction} and \cite[Theorem 5.5]{Kuhn-construction} respectively. Then, $f$ maps trivially into $H_*QS^0$ under $h$.\\
(ii) Similar vanishing results hold for the elements that are represented by certain triple Toda brackets \cite[Lemma 4.6, Theorem 4.8]{Kuhn-construction}.\\
($p>2$) Let $f$ be an element of the family considered by Hunter and Kuhn \cite[Corollary 3.8]{HunterKuhn} detected by $dh_0h_i$ in the Adams spectral sequences. Then, $h(f)=0$ for $h:{_p\pi_*^s}\simeq{_p\pi_*}QS^0\to H_*QS^0$.
%whose constructed is proposed by Kuhn \cite[Theorem 4.4, Example 4.5]{Kuhn-construction}, \cite[Theorem 4.8, Example 4.9]{Kuhn-construction}, \cite[Theorem 5.3, Exampple 5.4]{Kuhn-construction}, \cite[Theorem 5.5, Example 5.6]{Kuhn-construction}, \cite[Theorem 5.7]{Kuhn-construction} or by Hunter and Kuhn \cite[Corollary 3.8]{HunterKuhn}. Then, $f$ maps trivially under
\end{thm}

The proof of this is evident from the construction of the families. So, we omit the proof.\\

As an example where a factorisation such as required by Theorem \ref{main1} is not known, we prove a vanishing result for elements of ${_2\pi_*^s}$ which factor through certain Brown-Gitler spectra; this will include a family due to Lin \cite{Lin-BG}.

\begin{thm}\label{BrownGitlervanishing}
Suppose $r$ is even. Then for any $f:S^{2r}\to\Sigma^rB(r/2)\to S^0$ we have $h(f)=0$ where $h:{_2\pi_{2r}^s}\simeq{_2\pi_{2r}QS^0}\to H_{2r}QS^0$ is the unstable Hurewicz homomorphism.
\end{thm}

The following this is immediate.

\begin{crl}\label{Lin-BGfamily}
The elements of Lin's family in ${_2\pi_{2^{i+1}+2^{j+1}}^s}$ detected by $h_2(h_{i+1}h_j^2+h_i^2h_{j+1})$ in the ASS \cite{Lin-BG} map trivially under the unstable Hurewicz homomorphism.
\end{crl}

\begin{proof}
For $i,j$ with $4\leqslant i<j-1$ Lin has constructed this family as a composition $S^{2r}\to \Sigma^rB(r/2)\to S^0$ where $r=2^i+2^j$. Since $r$ is even, then the result follows from \ref{BrownGitlervanishing}.
\end{proof}

Next, we consider the image of the unstable Hurewicz homomorphism when restricted to submodule of decomposable elements in $\pi_*^s$.

\begin{thm}\label{main0}
The image of $h:\pi_*^s\simeq{\pi_*}QS^0\to H_*(QS^0;\Z)$ when restricted to the submodule of decomposable elements is given by
$\Z\{h(\eta^2),h(\nu^2),h(\sigma^2)\}$.
\end{thm}

%This theorem can be proved either integrally by applying Freduenthal's theorem which immediately implies $p$-local versions, or can be proved $p$-locally as a corollary of Theorem \ref{main1} at any prime $p$ which then will imply integral version.
{The integral result, immediately, implies the $p$-local version; the image of $h:{_p\pi_*^s}\simeq{_p\pi_*}QS^0\to H_*(QS^0;\Z/p)$ when restricted to the submodule of decomposable elements is given by the $\Z/p$-vector space $\Z/p\{h(\eta^2),h(\nu^2),h(\sigma^2)\}$. At the prime $p=2$,} the above theorem has a variant on the level of ASS due to Hung and Peterson \cite[Proposition 5.4]{HungPeterson} as follows. If $\varphi_k:\mathrm{Ext}^{k,k+i}_A(\Z/2,\Z/2)\to (\Z/2\otimes D_k)_i^*$ denotes the Lannes-Zarati homomorphism, then ${\varphi:=\oplus_k\varphi_k}$ does vanish on decomposable classes when $k>2$; here the authors take ${\varphi}$ as an algebraic approximation to the unstable Hurewicz homomorphism which seems to imply our Theorem \ref{main0} at the prime $2$. However, the relation between ${\varphi}$ and $h$ is not of that linear type. After first preprint versions of this work was posted on arxiv, Nick Kuhn kindly pointed out that, theoretically, it is possible to have the vanishing of the Lannes-Zarati homomorphism on a permanent cycle $c$ converging to an element $f$ with $h(f)\neq 0$ (see also \cite{Hung-erratum}). So, for the purpose of verifying Curtis conjecture, our result seems to provide a stronger evidence, and our proof is more geometric and quick.\\

As an application of this latter observation, we provide a computation of spherical classes in certain finite loop spaces on spheres. Let's first start with a conjecture due to Eccles to which our next results are related.

\begin{conj}\label{Ecclesconj}
(\textrm{Eccles conjecture}) Let $X$ be a path connected $CW$-complex with finitely generated homology. For $n>0$, suppose $h(f)\neq 0$ where ${_2\pi_n^s}X\simeq{_2\pi_n}QX\to H_*QX$ is the unstable Hurewicz homomorphism. Then, the stable adjoint of $f$ either is detected by homology or is detected by a primary operation in its mapping cone.
\end{conj}

Note that the stable adjoint of $f$ being detected by homology means that $h(f)\in H_*QX$ is stably spherical, i.e. it survives under homology suspension $H_*QX\to H_*X$ induced by the evaluation map $\Sigma^\infty QX\to \Sigma^\infty X$ given by the {stable} adjoint of the identity $QX\to QX$.

The above conjectures make predictions about the image of spherical classes in the image of unstable Hurewicz homomorphism $h:{_2\pi_*^s}X\simeq{_2\pi_*}QX\to H_*QX$. By Freudenthal suspension theorem, for any $f\in\pi_nQX$, depending on the connectivity of $X$, we may find some nonnegative integer $i$ so that $f$ does pull back to $\pi_n\Omega^i\Sigma^iX$. {Note that there exist obvious commutative diagrams
$$\xymatrix{
\pi_n\Omega^i\Sigma^iX\ar[r]^-h\ar[d] & H_n\Omega^i\Sigma^iX\ar[d]\\
\pi_nQX\ar[r]^-h                      & H_nQX.}$$
So it is then natural to look for spherical classes in $H_*\Omega^i\Sigma^iX$. Our next result consider the special case of $X$ being a sphere. For various values of $d$ and $k$, the spaces $\Omega^dS^{d+k}$ determine a lattice with positive integer coordinates. It is easier to state our next result in terms of this lattice. We have the following.}

\begin{thm}\label{main}
Consider the following lattice where $(a,b)$ corresponds to $H_*\Omega^aS^b$
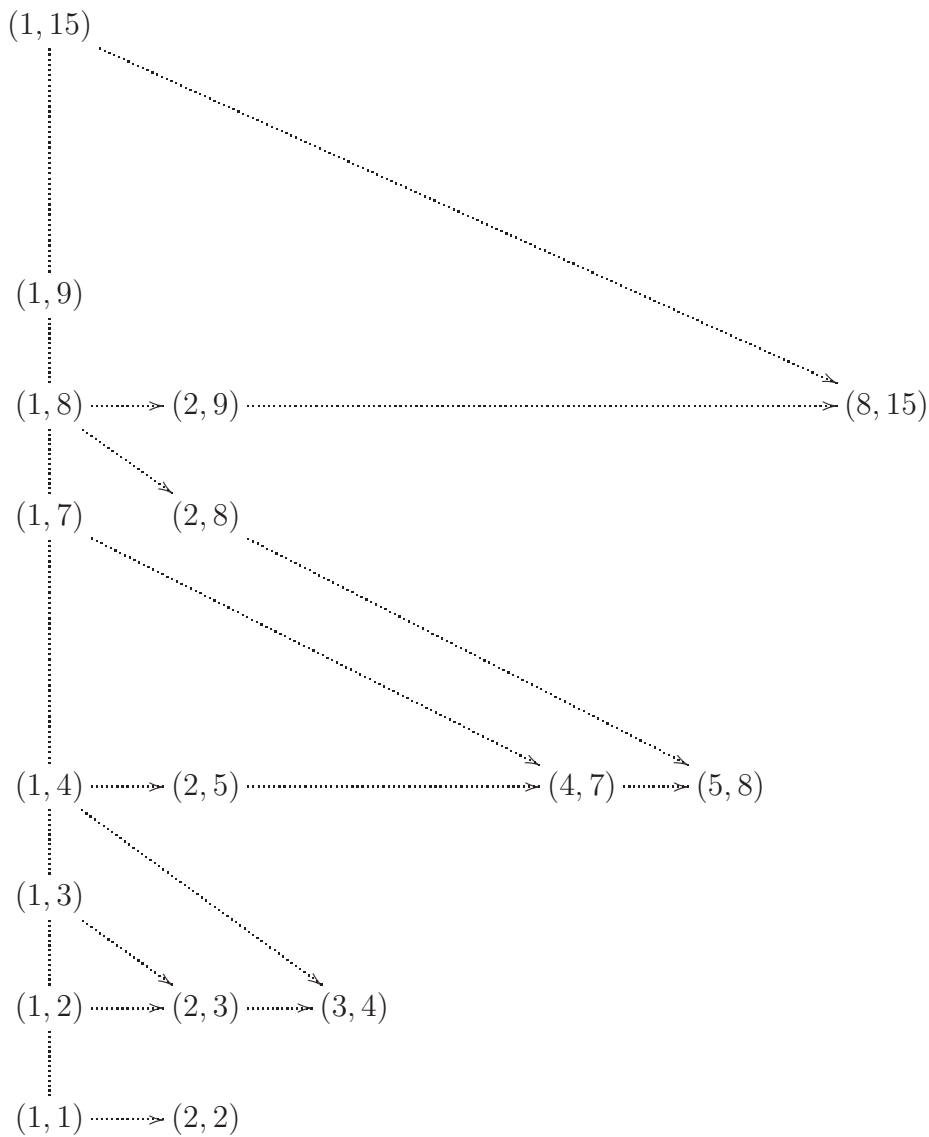
$$\xymatrix{
(1,15)\ar@{.}[ddd]\ar@{.>}[ddddrrrrrr]  \\
\\
\\
(1,9)\ar@{.}[d]\\
(1,8)\ar@{.>}[r]\ar@{.}[d]\ar@{.>}[rd] & (2,9)\ar@{.>}[rrrrr] & & & & & (8,15)\\
(1,7)\ar@{.>}[dddrrrr]\ar@{.}[ddd]     & (2,8)\ar@{.>}[rrrrddd]\\
\\
\\
(1,4)\ar@{.>}[r]\ar@{.}[d]\ar@{.>}[ddrr]          & (2,5)\ar@{.>}[rrr]&&& (4,7)\ar@{.>}[r] & (5,8)\\
(1,3)\ar@{.>}[rd]\ar@{.}[d]\\
(1,2)\ar@{.>}[r]\ar@{.}[d]& (2,3)\ar@{.>}[r] & (3,4)\\
(1,1)\ar@{.>}[r]          & (2,2)
}$$
The following table completely determines the spherical classes in the above lattice.
\begin{center}
\begin{tabular}{|c|c|c|c|c|c|c|c}
\hline
$b-a$ & $a$             & $b$              & \textrm{spherical classes arise from}\\
\hline
$0$   & $2$             & $2$              & $\eta,\eta^2$\\
\hline
$1$   & $2$             & $3$              & $\eta,b$\\
\hline
$1$   & $1$             & $2$              & $\eta,b$\\
\hline
$1$   & $3$             & $4$              & $\nu,\eta,b$\\
\hline
$2$   & $1$             & $3$              & $b$\\
\hline
$3$   & $\leqslant 4$   & $a+3$            & $\nu,b$\\
\hline
$3$   & $5$             & $8$              & $\sigma,\nu,b$\\
\hline
$7-a$ & $<4$            & $7$              & $b$\\
\hline
$7$   & $\leqslant 8$   & $a+7$            & $\sigma,b$\\
\hline
$15-a$& $<8$            & $15$             & $b$\\
\hline
$4-a$ & $<3$            & $4$              & $\nu,b$\\
\hline
$8-a$ & $<5$            & $8$              & $\sigma,b$\\
\hline
\textrm{other cases}    &    &                   & $b$\\
\hline
\end{tabular}
\end{center}
Here, $b$ corresponds to the inclusion of the bottom cell in the related loop space, that is $b:S^k\to \Omega^dS^{d+k}$. The diagonal arrows correspond to (iterated) adjointing down, and horizontal arrows correspond to (iterated) stablisation map $E$. Both of this operations, preserve spherical classes.
\end{thm}

%Our computations determine all spherical classes in homology of the spaces on the edges as well as inside the triangles. It also determines spherical classes on the edges of the bottom square. However, it does not completely determine spherical classes left outside the triangles or the square such as $\Omega^3S^3$.
The stablisation map $\Omega^dS^{d+k}\to \Omega^{d+l}S^{d+k+l}$ induces a monomorphism in homology, hence the spherical classes computed in the above theorem survive under this map. However, having determined spherical classes in $H_*\Omega^dS^{d+k}$ does not completely determine spherical classes in $H_*\Omega^{d+l}S^{d+k+l}$; for some more discussions see Corollary \ref{Curtis-1}.\\

The techniques that we employ in this paper, mainly geometric. In a sequel, we have used homology of James Hopf maps to completely eliminate spherical classes in $H_*\Omega^dS^{n+d}$ with $n>0$ and $d<4$ \cite{Zare-Els-1}.\\

\tb{Completion vs. localisation.} Often, we deal with maps $S^n\to QS^0$ with $n>0$. So, we may safely replace $QS^0$ with its basepoint component $Q_0S^0$. The space $Q_0S^0$ is a path connected infinite loop space with $\pi_iQ_0S^0\simeq\pi_i^s$ being finite Abelian groups. So, by generalities on localisation and completion at a prime $p$ \cite{BousfieldKan}, noting that completion at $p$ is just $\Z_p$-localisation, the effect on either localisation or completion at $p$ on homology will be the same.\\

\tb{Acknowledgements.} I am grateful to Nick Kuhn for his comments and notes on earlier versions of this document. I am grateful to an anonymous referee for his/her helpful comments; in particular for suggesting a simple proof of Proposition \ref{equiv1} as quoted, as well as Theorem \ref{main1} which allowed quick proof of some of results in this paper; in earlier versions of this paper we used Kuhn's generalised Kahn-Priddy theorem to derive Proposition \ref{equiv1} $p$-locally where the current version of this proposition offers an integral equivalence. This latter also allowed to prove some results of \cite{Za-ideal} for a wider range of spectra and not only suspension spectra.

\section{Freudenthal Theorem and the unstable Hurewicz homomorphism}
In this section, we prove Theorem \ref{main1}. We fix some notation that is used below. We write $h^s:\pi_*E\to H_*E$ for the {(stable)} Hurewicz homomorphism, {that is the Hurewicz homomorphism on the spectrum level,} and $h:\pi_*E\simeq\pi_*\Omega^\infty E\to H_*\Omega^\infty E$ for the unstable Hurewicz homomorphism which {satisfy} $h^s=\epsilon_*\circ h$.

\begin{proof}[{Proof of Theorem \ref{main1}}]
(i) If $h(f)\neq 0$ then $(\Sigma^\infty f)_*\neq 0$, i.e. it maps nontrivially under the Hurewicz map $h^{{s}}:\pi_*\Sigma^\infty QS^0\to H_*\Sigma^\infty QS^0$ is nonzero in homology. Consequently $(\Sigma^\infty f_E)_*\neq 0$. On the other hand, there is a commutative diagram
$$\xymatrix{
\pi_n\Sigma^\infty\Omega^\infty E\ar[r]^-{\epsilon_*}\ar[d]_-{h^s}  & \pi_nE\ar[d]^-{h^s}\\
H_n\Sigma^\infty\Omega^\infty E\ar[r]^-{\epsilon_*} & H_nE
}$$
where the rows are isomorphisms. However, this leads to a contradiction as
$$0\neq \epsilon_*h^s(\Sigma^\infty f_E)=h^s\epsilon_*(\Sigma^\infty f_E)=0.$$
This completes the proof {in the integral case. The $p$-local version follows, with analogous techniques, replacing $E$ with its $p$-local version $E_{(p)}$, together with the following facts: (1) the $(2r+1)$-equivalence $\epsilon:\Sigma^\infty \Omega^\infty E\to E$ of Theorem \ref{equiv1} is an integral one, so it holds after $p$-localisation as well; (2) $\Omega^\infty$ and localisation functors commute
\cite[Theorem 1.1]{Bousfield-Klocalisation}; (3) infinite loop spaces are nilpotent in the sense of Bousfield, so both of homotopy and homology functors commute with localisation \cite[Chapter V, Theorem 3.1]{BousfieldKan}. The proof in the $p$-complete setting is similar.}\\
(ii) Write $\widetilde{f}:S^n\stackrel{\widetilde{f_E}}{\lra} E\stackrel{c}{\lra} S^0$ for the adjoint of $f$. Apply the $S$-duality functor $D$ to obtain
$$D(\widetilde{f}):S^n\stackrel{\Sigma^n D(c)}{\lra} \Sigma^nD(E)\stackrel{\Sigma^nD(\widetilde{f_E})}{\lra}S^0.$$
The result now follows from part (i) and the fact that the elements of $\pi_*^s$ are self $S$-dual. Note that {as $E$ and $S^0$ are finite CW-spectra, then $c_*=0$ if and only if ${D(c)}_*=0$, or equivalently $h(D(c))=0$,} by \cite[Lemma A.3]{Kuhn-construction}.
\end{proof}

\section{Hurewicz Image of some infinite families}
In this section, we compute image of some infinite families under the unstable Hurewicz homomorphism. Most of these follow immediately from Theorem \ref{main1}. As an example where Theorem \ref{main1} does not apply, we will compute image of certain family in ${_2\pi_{2^{i+1}+2^{j+1}}^s}$ due to Lin \cite{Lin-BG}.

\subsection{Snaith splitting and related spectra:Proof of Theorem \ref{infinitefamily}}
Recall that for a path connected space $X$, the space $\Omega^k\Sigma^kX$ is filtered by certain spaces $C^k_n(X)$ \cite{May-G}. By Snaith there is an equivalence of suspension spectra
\begin{equation}\label{snaithsplitting}
\Sigma^\infty\Omega^k\Sigma^kX\simeq\bigvee_{n=1}^{+\infty}\Sigma^\infty D_{k,n}X
\end{equation}
where $D_{k,n}X$ is the cofibre of $C^k_{n-1}(X)\to C^k_n(X)$ given by $F(\R^k,n)\ltimes_{\Sigma_n}X^{\wedge n}$ \cite{Snaith}. By construction, $D_{k,n}X$ is a genuine space for any $k$ and $n$ which has its bottom cell in dimension $rn$ if $X$ has its bottom cell in dimension $r$. For a prime $p$ and $t\in\{0,1\}$, following Hunter and Kuhn \cite{HunterKuhn}, at the prime $p$ define
$$T(2n+t)=\Sigma^{2pn+2t}D(D_{2,pn+t}S^1)$$
where $D(-)$ denotes the $S$-duality functor. For $S^1$ having its bottom cell in dimension $1$, the spectrum $T(2n+t)$ has its top cell in dimension $pn+t$. These spectra, at least on homological level, are related to Brown-Gitler spectra. For instance, at the prime $p=2$, there is a homotopy equivalence $\Sigma^\infty D_{2,n}S^1\to\Sigma^nB([n/2])$ where $[-]$ is the integer-part function \cite[Theorem B]{BrownPeterson}. We are now able to prove Theorem \ref{infinitefamily}.

\begin{proof}[Proof of Theorem \ref{infinitefamily}]
First, let $p=2$. Then,
\begin{itemize}
\item For $i\geqslant 3$, $\eta_i$ is constructed as a composition $S^{2^i}\to D_{2,2^i-3}S^7\to S^0$ \cite{Ma} (working in $2$-complete session, we may construct $\eta_i$ as a certain composition $S^{2^i}\to\R P^{2^{i-3}}\to S^0$ \cite[Remark 2.2 and Theorem 4.1]{Kuhn-construction} or as a composition $S^{2^i}\stackrel{}{\to}T(2^{i-3})\to S^0$ \cite[Proposition 2.1(1), Theorem 3.4]{HunterKuhn});
\item Lin's family of elements detected by $c_1h_i$ is constructed as a composition $S^{2^i+18}\to\Sigma^{2^i}P_1^4\to S^0$ (working in $2$-complete session y
      construct these as a composition $S^{2^i+18}\to \R P^{2^{i-3}}\to S^0$ \cite[Section 6]{Kuhn-construction});
\item For $i>11$, Lin's family of elements for $i>11$ is as a composition
$$S^{2^{i+1}+37}\stackrel{\{\overline{d_1}\}}{\longrightarrow}\Sigma^{2^{i+1}+1}\C P^2\to S^0.$$
which is detected by $h_i^2h_3d_1$ in the Adams spectral sequence \cite{Lin-7} (see also \cite[Theorem 7.3]{HunterKuhn})
\item working $2$-complete $\tau_{i-1}$ can be constructed as a composition $S^{2^{i}+1}\stackrel{}{\to}\Sigma T(2^{i-3})\to S^0$
\cite[Proposition 2.1(1), Theorem 3.6(1)(b)]{HunterKuhn}.\\
\end{itemize}
Taking either of the compositions for the afore mentioned families, the dimensional reasons together with Theorem \ref{main1} show that these elements map trivially under $h:{_2\pi_*^s}\simeq{_2\pi_*}QS^0\to H_*QS^0$.\\
When $p$ is odd, appealing to \cite[Proposition 2.1(2), Theorem 3.6(2)(b)]{HunterKuhn}, the result follows by a similar reasoning.
\end{proof}

Let's note that it is possible to compute $h(\tau_i)$ using the relation between the Hurewicz homomorphism and homotopy operations arising from ${_2\pi_*^s}D_2S^n$ as done in \cite[Proposition 0]{Za-ideal} which does not need $2$-completion.

\subsection{Factorisation through certain $B(n)$ spectra}
This section is devoted to the proof of Theorem \ref{BrownGitlervanishing}.\\

It is known that, at the prime $p=2$, there is a homotopy equivalence $\Sigma^\infty D_{2,n}S^1\to\Sigma^nB([n/2])$ where $[-]$ is the integer-part function
\cite[Theorem B]{BrownPeterson}. Here, the spaces $D_{2,n}S^1$ are stable summands of  Snaith splitting (\ref{snaithsplitting}) for $\Omega^2S^3$. The composition $S^{2r}\to \Sigma^rB(r/2)\to S^0$ of maybe written as
$$S^{2r}\to \Sigma^\infty D_{2,r}S^1\to S^0$$
with the map on the left belonging to ${_2\pi_{2r}^s}D_{2,r}S^1$. By construction, the complex $D_{2,r}S^1=F(\R^2,r)\ltimes_{\Sigma_r}(S^1)^{\wedge r}$ is a genuine complex with its bottom cell in dimension $r$. By Freudenthal theorem, the stable adjoint of the above composition admits a factorisation as
$$S^{2r}\to\Omega\Sigma D_{2,r}S^1\to QD_{2,r}S^1\to QS^0$$
where the pull back $S^{2r}\to\Omega\Sigma D_{2,r}S^1$ is not necessarily unique. The result now follows from the following lemma.

\begin{lmm}
For any map $f:S^{2r}\to\Omega\Sigma D_{2,r}S^1$ we have $f_*=0$.
\end{lmm}

\begin{proof}
According to James \cite{James-reducedproduct} for a path connected space $X$, there is an isomorphism of algebras $H_*(\Omega\Sigma X;k)\simeq T_k(\widetilde{H}_*(X;k))$ where $T_k(-)$ is the free tensor algebra over $k$, and $k$ is an arbitrary PID. We wish to show that there is no spherical class in $H_{2r}\Omega\Sigma D_{2,r}S^1\simeq T_{\Z/2}(\widetilde{H}_*D_{2,r}S^1)$. The space $D_{2,r}S^1$ has its bottom cell in dimension $r$, hence first nontrivial homology living in dimension $r$. The algebra structure of $H_*\Omega\Sigma D_{2,r}S^1$ then implies that $H_{2r}\Omega\Sigma D_{2,r}S^1$ is the vector space generated by typical elements $\alpha_r\otimes\beta_r$ and $\gamma_{2r}$ where $\alpha_r,\beta_r\in H_rD_{2,r}S^1$ and $\gamma_{2r}\in H_{2r}D_{2,r}S^1$ are arbitrary elements. Next, note that the homology $\widetilde{H}_*D_{2,r}S^1$ corresponds to the part of homology $H_*\Omega^2S^3$ of height $r$, and consequently $H_*\Omega\Sigma D_{2,r}S^1$ will be the free graded associative algebra generated by elements of $H_*\Omega^2S^3$ that are of height $r$
\cite[Theorem 9.4.5]{Ravenel-Orange}. The height $r$ elements of $H_*\Omega^2S^3$ are precisely those elements that map nontrivially under the projection
$(p_r)_*:\widetilde{H}_*\Omega^2S^3\to \widetilde{H}_*D_{2,r}S^1$ provided by Snaith splitting $\Sigma^\infty\Omega^2S^3\simeq\bigvee_{k=1}^{+\infty}\Sigma^\infty D_{2,k}S^1$. We recall the description of $H_*\Omega^2S^3$. At $p=2$, there is an isomorphism of algebras
$$H_*\Omega^2S^3\simeq\Z/2[x_1,x_3,x_7,x_{15},\ldots]$$
where $x_j\in H_j\Omega^2S^3$ \cite[Proposition 9.4.1]{Ravenel-Orange}. For an arbitrary class $x_{2^{i_1}-1}^{j_1}\cdots x_{2^{i_n}-1}^{j_n}\in\widetilde{H}_*\Omega^2S^3$ with $j_t\geqslant 0$ we write $\overline{x_{2^{i_1}-1}^{j_1}\cdots x_{2^{i_n}-1}^{j_n}}$ for its image under $(p_r)_*$. The height function $h:H_*\Omega^2S^3\to\Z_+$ is defined by setting $h(x_{2^i-1})=2^{i-1}$ and $h(\eta\xi)=h(\eta)+h(\xi)$ for any $\eta,\xi\in H_*\Omega^2S^3$. The definition of hight function implies that for a monomial $x_{2^{i_1}-1}^{j_1}\cdots x_{2^{i_n}-1}^{j_n}$ we have
$$h(x_{2^{i_1}-1}^{j_1}\cdots x_{2^{i_n}-1}^{j_n})=\sum_{t=1}^n h(x_{2^{i_t}-1}^{j_t})=\sum_{t=1}^n j_th(x_{2^{i_t}-1})=\sum_{t=1}^n j_t2^{i_t-1}.$$
By inspection, one can see that the only height $r$ class as above, living in dimension $r$ is $x_1^r$ giving rise to a nontrivial class $\overline{x_1^r}\in\widetilde{H}_rD_{2,r}S^1$. From this class, we get the nontrivial class $\overline{x_1^r}\otimes\overline{x_1^r}\in H_{2r}\Omega\Sigma D_{2,r}S^1$. Moreover, for any monomial $x_{2^{i_1}-1}^{j_1}\cdots x_{2^{i_n}-1}^{j_n}\in H_*\Omega^2S^3$ of height $r$, requiring $\dim(x_{2^{i_1}-1}^{j_1}\cdots x_{2^{i_n}-1}^{j_n})=\sum_{t=1}^n j_t(2^{i_t}-1)=2r$, multiplying the above equality for height by $2$, implies that $\sum j_t=0$, and consequently $j_t=0$ for all $t=1,\ldots,n$. That is we don't have any nontrivial monomial $x_{2^{i_1}-1}^{j_1}\cdots x_{2^{i_n}-1}^{j_n}$ of dimension $2r$ and height $r$. This shows that
$H_{2r}\Omega\Sigma D_{2,r}\simeq\Z/2\{\overline{x_1^r}\otimes\overline{x_1^r}\}$.\\
Now, if $f:S^{2r}\to\Omega\Sigma D_{2,r}S^1$ is any map with $f_*\neq 0$ then $h(f)=\overline{x_1^r}\otimes\overline{x_1^r}$. Consequently, the composition
$S^{2r}\to\Omega\Sigma D_{2,r}S^1\to QD_{2,r}S^1$ will behave similarly in homology. By \cite[Proposition 5.8]{AsadiEccles} this latter implies that the stable adjoint of $f$, $S^{2r}\to \Sigma^\infty D_{2,r}S^1$ is detected by $Sq^{r+1}$ on the bottom dimensional cohomology of its stable mapping cone, i.e. $Sq^{r+1}\overline{x_1^r}=g_{2r+1}$ where by abuse of notation $\overline{x_1^r}\in H^rC_f\simeq H^{r}D_{2,r}S^1\simeq\Z/2$ denotes a generator of $r$-dimensional cohomology and $g_{2r+1}\in H^{2r+1}C_f$ denotes a generator coming from the $2r+1$-cell attached by $f$. We show that the equality $Sq^{r+1}\overline{x_1^r}=g_{2r+1}$ cannot happen in $H^*C_f$. Since $r$ is even then $Sq^{r+1}=Sq^1Sq^r$. Moreover, by \cite[Theorem 1.3]{BrownGitler} $H^*B(r/2)\simeq 0$ for $*>r-\alpha(r/2)\geqslant r-1$, where $\alpha(n)$ is the number of $1$'s in binary expansion of $n$; in particular $H^rB(r/2)\simeq 0$.  Consequently, $Sq^r$ acts trivially on the zero dimensional generator of $B(r/2)$. By stability of Steenrod operations this implies that $Sq^r\overline{x_1^r}=0$ in $H^*D_{2,r}\simeq H^* C_f$ (in dimensions $*<2r+1$). This contradicts $Sq^{r+1}\overline{x_1^r}=g_{2r+1}$ as $Sq^{r+1}\overline{x_1^r}=Sq^1Sq^r\overline{x_1^r}=Sq^10=0$. Hence, $f_*=0$ and consequently $h(f)=0$.
\end{proof}

\section{The unstable Hurewicz image of decomposable elements in ${_2\pi_*^s}$}
As we said earlier, we are interested in Freudenthal theorem as it is an integral result and comes with no localisation. This section is devoted to the proof of Theorem \ref{main0}.

\begin{proof}[Proof of Theorem \ref{main0}]
For $f\in\pi_i^s$ and $g\in\pi_j^s$, the product $gf$ in ${\pi_*^s}$ is determined by the composition of stable maps $S^{i+j}\stackrel{f}{\to}S^j\stackrel{g}{\to}S^0$. {For $i>j$, we have $i+j>2j$ which implies that $i+j\geqslant 2j+1$. The result follows from Theorem \ref{main1}(ii). The case $i<j$ follows either by the same reasoning or from commutativity of $\pi_*^s$.}\\
 % whose stable adjoint is given by $S^{i+j}\lra QS^j\lra QS^0$. For $i\neq j$, $f:S^{i+j}\to S^j$ maps nontrivially under the stable Hurewicz homomorphism $h^s:\pi_{i+j}^sS^j\to H_{i+j}(S^j;\Z)$. By, Theorem \ref{main1}, $S^{i+j}\stackrel{f}{\to}S^j\stackrel{g}{\to}S^0$ is trivial in homology, that is $h(fg)=0$ for $h:\pi_{i+j}QS^0\to H_{i+j}QS^0$.\\
%Since $i+j>0$ then we may replace $QS^0$ with its base point component $Q_0S^0$. It is known that $\Omega^\infty$ and localisation commute \cite[Theorem 1.1]{Bousfield-Klocalisation}. If $h(gf)\in H_*(Q_0S^0;\Z)$ is nonzero, then there exists a prime $p$ that after localisation at $p$ we still have a nonzero element. Note that infinite loop spaces are nilpotent in the sense of Bousfield, so by \cite[Chapter V, Theorem 3.1]{BousfieldKan} $h(gf)$ is nonzero in $H_*(Q_0S^0;\Z/p)$. But, for $i\neq j$ this contradicts Theorem \ref{main1}. So, $h(fg)=0$.\\
Next, suppose $i=j=n$. {As elements of $\pi_*QS^0$, the product $gf$ is given by the composition $S^{2n}\to QS^n\to QS^0$.} By Freudenthal's theorem, the stale adjoint of $gf$ factors as
$$S^{2n}\stackrel{f'}{\lra} \Omega S^{n+1}\lra QS^n\lra QS^0$$
{for some $f':S^{2n}\to\Omega S^{n+1}$ which is not necessarily unique.} Now, $h(fg)\neq 0$ implies that $h(f')\neq 0$, and consequently $h(f')=lx_n^2$ where $x_n\in H_n(\Omega S^{n+1};\Z)$ is a generator, and $l\in\Z-\{0\}$. It is  well known (see for example \cite[Proposition 6.1.5]{Harper}) {that} $h(f')\neq 0$ if and only if $f$ is detected by the unstable Hopf invariant {(as defined in many classic books such as \cite[Chapter 4]{MosherTangora})}. {Similarly, $g$ is detected by the unstable Hopf invariant. This implies that $n+1$ has to be even, i.e. $n$ must be odd.}\\
{If $f$ is of even unstable Hopf invariant then $f'=l[\iota_{n+1},\iota_{n+1}]+f''$ for some $f''\in \ker(H^{\textrm{unstable}})$ and some even nonzero integer $l$. Here, $H^{\textrm{unstable}}:\pi_{2n}\Omega S^{n+1}\simeq\pi_{2n+1}S^{n+1}\to\Z$ is the unstable Hopf invariant and $[\iota_{n+1},\iota_{n+1}]$ is the Whitehead product. Since the Whitehead product $[\iota_{n+1},\iota_{n+1}]$ belongs to the kernel of suspension map $E_*:\pi_*\Omega S^{n+1}\to\pi_*\Omega^2S^{n+2}$, it follows that $f=E^\infty f'=E^\infty f''$ where $E^\infty:\Omega S^{n+1}\to QS^n$ is the stabilisation map. Consequently, this would imply that in homology $f'':S^{2n}\to\Omega S^{n+1}$ sends $x_{2n}$ to $lx_{n}^2$, that is $f''$ is of unstable Hopf invariant $l$ which possible only if $l=0$. But, this is a contradiction. Hence, $f$ (and similarly $g$) cannot be of even unstable Hopf invariant.\\
Therefore, $f$ should be of odd unstable Hopf invariant. In this case, reduction mod $2$ implies that $f$ would be an odd multiple of $\eta$, $\nu$, or $\sigma$ which determines $n\in\{1,3,7\}$ by the Hopf invariant one result \cite{Adams-Hopfinv}. Similarly, we would have $g$ as an odd multiple of Hopf invariant one elements. On the other hand it is known that $\eta^2$, $\nu^2$, and $\sigma^2$ maps nontrivially under $h:\pi_*^s\simeq\pi_*QS^0\to H_*(QS^0;\Z)$. This completes the proof.}
\end{proof}

%Note that it is possible to prove the whole theorem, integrally and without any localisation, using Freudenthal's theorem as done in \cite{Za-ideal}. But, we wish to present this as an application of Theorem \ref{main1}. After all, a finite Abelian group is determined by its $p$-local behaviour when $p$ varies among prime numbers!

\subsection{Spherical classes in some finite loop spaces: Proof of Theorem \ref{main}}
{The main goal of this section, as well as the following section, is to provide an application of Theorem \ref{main0} by proving Theorem \ref{main}. This allows to derive positive evidence for the truth of Curtis and Eccles conjectures when restricted to finite loop spaces. For this reason, during this section, as well as the following section, we will only work at the prime $2$ writing $\pi_*$ and $H_*$ for the $2$-component of $\pi_*$ and $H_*(-;\Z/2)$, respectively. We only note that similar techniques might be applied for the odd primary case as well.}

We will consider using James fibrations $S^n\stackrel{E}{\to}\Omega S^{n+1}\stackrel{H}{\to}\Omega S^{2n+1}$ referring to $E$ as the suspension and $H$ as the second James-Hopf map. We also write $E^\infty:X\to QX$ for the inclusion which induces $\pi_*X\to\pi_*QX\simeq\pi_*^sX$ sometimes referring to it as the stablisation. We also may write $E^k$ for the iterated suspension $X\to\Omega^k\Sigma^kX$. We shall use $H_*\Omega\Sigma X\simeq T(\widetilde{H}_*X)$ where $X$ is any path connected space, $T(-)$ is the tensor algebra functor and $\widetilde{H}_*$ denotes the reduced homology. We shall write $\sigma_*:H_*\Omega X\to H_{*+1}X$ for homology suspension induced by the evaluation $\Sigma\Omega X\to X$ recalling that it kills decomposable elements in the Pontrjagin ring $H_*\Omega X$. We begin with the following.

\begin{thm}
(i) The only spherical classes in $H_*\Omega^2S^3$ live in dimensions $1$ and $2$, arising from the identity $S^3\to S^3$ and $\Sigma\eta:S^4\to S^3$. There are no spherical classes in $H_*\Omega S^3$ other than the bottom dimensional class given  by $S^2\to\Omega S^3$.\\
(ii) The only spherical classes in $H_*\Omega^4S^7$ live in dimensions $3$ and $6$, arising from the identity $S^7\to S^7$ and $\Sigma^3\nu:S^{10}\to S^7$. There are no spherical classes in $H_*\Omega^i S^7$ for $i<4$ other than $S^{3+i}\to \Omega^{4-i}S^7$ adjoint to the identity $S^7\to S^7$ that corresponds to the inclusion of the bottom cell.\\
(iii) The only spherical classes in $H_*\Omega^8S^{15}$ live in dimensions $7$ and $14$, arising from the identity $S^{15}\to S^{15}$ and $\Sigma^7\sigma:S^{22}\to S^{15}$. There are no spherical classes in $H_*\Omega^i S^{15}$ for $i<8$ other that the one arising from the inclusion of the bottom cell corresponding to the identity on $S^{15}$.
%(iv) There is no class in $\pi_*B(r)$ that is detected by homology, unless $r=0$.
\end{thm}

\begin{proof}
Let $k\geqslant 0$. Consider the stablisation $S^{i+k}\to QS^{i+k}$ and by abuse of notation write $E^\infty:\Omega^iS^{i+k}\to\Omega^iQS^{i+k}=QS^k$ for its $i$-fold loop. Recall that $E^\infty_*$ induces a monomorphism in homology. We also write $E:S^k\to\Omega S^{k+1}$ for the suspension map, that induces the suspension homomorphism $\pi_*S^k\to\pi_{*}\Omega S^{k+1}\simeq\pi_{*+1}S^{k+1}$. We proceed as follows.\\
(i) Obviously, the inclusion of the bottom cell $S^1\to\Omega^2S^3$ is nontrivial in homology. Let $f:S^n\to\Omega^2S^3$, $n>1$, be any map with $f_*\neq 0$. First, note that the Hopf fibration $\eta:S^3\to S^2$ induces a homotopy equivalence $\Omega^2\eta:\Omega^2S^3\to\Omega^2_0S^2$ where $\Omega^2_0S^2$ is the base point component of $\Omega^2S^2$. In particular, $\Omega^2\eta$ is an isomorphism in homology. Consequently, the composition
$((\Omega^2E^\infty)(\Omega^2\eta)f)_*\neq 0$. On the other hand, note that we have a commutative diagram
$$\xymatrix{
\Omega^2S^3\ar[r]^-{\Omega^2 E^\infty}\ar[d]^-{\Omega^2\eta}    & QS^1\ar[d]^-{\Omega^2 \eta}\\
\Omega^2_0S^2\ar[r]^-{\Omega^2 E^\infty}                        & Q_0S^0
}$$
where the right vertical map is obtained from $\eta:QS^3\to QS^2$ and we write $Q_0S^0$ for the base point component of $QS^0$. This implies that $(\Omega^2\eta)_*((\Omega^2E^\infty)f)_*\neq 0$. This implies that for $E^\infty f$ and $\eta$ as elements of $\pi_*^s$, we have $h(\eta(E^\infty f))\neq 0$ where $h$ is the unstable Hurewicz homomorphism. By Theorem \ref{main0} if $n>0$ then $E^\infty f$ and $\eta$ are in the same dimension. That is $E^\infty f\in\pi_1^s$, consequently $n=2$. Therefore, $f\in\pi_2\Omega^2S^3\simeq\pi_4S^3\simeq\Z/2$, which implies that $f= E\eta$ viewing $\eta:S^3\to S^2$. Moreover, note that by Freudenthal's theorem $f$ pulls back to $\pi_3S^2$ allowing to consider $f$ as $f:S^2\to\Omega S^2$. The only $2$-dimensional class in $H_2\Omega S^2$ is given by $x_1^2$ which maps to $x_1^2\in H_2\Omega^2S^3$ which is a decomposable class, where $x_1\in\widetilde{H}_1S^1$ is a generator.\\
Moreover, suppose $f$ is any map $S^{n+1}\to\Omega S^3$ which is nontrivial in homology. By adjointing down, we have a map $f':S^n\to\Omega^2S^3$ which is nontrivial in homology. We have $f_*\sigma_*x=\sigma_*f'_*x$ for any homology class $x$. However, by the above computation, if $f'_*x\neq 0$ then $f'_*x=x_1^2$ which is a decomposable class that is killed under homology suspension. This contradicts the claim that $f_*\neq 0$ for $f:S^{n+1}\to\Omega S^3$. \\
(ii) The argument here is almost the same. We consider the Hopf fibration
$$\Omega S^7\to\Omega S^4\to S^3\to S^7\to S^4$$
noting that $H(\nu)=1$ where $H:\Omega S^4\to \Omega S^7$ is the second James-Hopf map. This provides us with a decomposition $\Omega S^4\to S^3\times\Omega S^7$. A choice of decomposition may be given by $\Omega S^7\stackrel{(*,1)}{\longrightarrow}S^3\times\Omega S^7 \stackrel{E+\Omega\nu}{\longrightarrow}\Omega S^4$ and $H:\Omega S^4\to \Omega S^7$ in the other direction where $+$ is the loop sum in $\Omega S^4$. The decomposition of $3$-fold loop spaces, in particular implies that there is monomorphism in homology given by $\Omega^4\nu:\Omega^4 S^7\to\Omega^4S^4$.\\
Note that $\widetilde{H}_n\Omega^4S^7\simeq 0$ for $n<3$. Similar to the previous part, the inclusion of the bottom cell $S^3\to\Omega^4S^7$, adjoint to the identity $S^7\to S^7$ is nontrivial in homology. Suppose $f:S^n\to\Omega^4S^7$ with $n>3$ and $f_*\neq 0$ is given. Then, similar as above, we have $((\Omega^3 E^\infty)(\Omega^3\nu)f)_*\neq0$. Consider the commutative diagram
$$\xymatrix{
S^7\ar[r]^-{E^\infty}\ar[d]^-\nu & QS^7\ar[d]^-{\nu}\\
S^4\ar[r]^-{E^\infty}            & QS^4}$$
and then loop it $4$-times. Moreover, using the splitting of $\Omega S^4$ as above, we may rewrite the quadruple loop of the left vertical arrow and obtain a commutative diagram of $4$-fold loop maps as
$$\xymatrix{
\Omega^4S^7\ar[r]^-{\Omega^4E^\infty}\ar[d]_-{\Omega^4(*,1_{S^7})} & QS^3\ar[dd]^-{\Omega^4\nu}\\
S^3\times\Omega^4S^7\ar[d]_-{\Omega^3(E+\Omega\nu)} \\
\Omega^4S^4\ar[r]^-{\Omega^4E^\infty}                     & QS^0}$$
noting that the composition of vertical arrows on the left is just $\Omega^4\nu$. By commutativity of the above diagram, we see that $h(\nu E^\infty f)\neq 0$. Consequently, by Theorem \ref{main0} as elements of $\pi_*^s$, $E^\infty f$ is in the same dimension as $\nu$ implying that $n=6$. Therefore, $f\in{_2\pi_6}\Omega^4S^7\simeq{_2\pi_{10}}S^7\simeq\Z/8$ and $f=E^3\nu$ where $E^3$ is the iterated suspension (strictly speaking, $f$ will be three fold suspension of an odd multiple of $\nu$ in ${_2\pi_3^s}$, but since we are working at the prime $2$, we take it as $\nu$). By Freudenthal suspension theorem it pulls back to $\pi_7S^4\simeq\pi_6\Omega S^4$. Hence, $f$ does factorise as $f:S^6\to\Omega S^4\to\Omega^4S^7$. The map $S^6\to\Omega S^4$ has square of a three dimensional class as its Hurewicz image. Hence, as previous part, eliminating classes in $H_*\Omega^iS^7$ with $i<4$ from being spherical.\\
(iii) Similar to the previous part, one has to use a decomposition $\Omega S^8\simeq S^7\times\Omega S^{15}$. We leave the rest to the reader.\\
%(iv) This follows from part (i) and Brown-Peterson $2$-local equivalence.
\end{proof}

The above has an application to the finite case of Curtis conjecture.

\begin{crl}\label{Curtis-1}
Suppose $f\in H_*\Omega^8S^8$ is a spherical class. Then, $f$ either corresponds to a spherical class in $\Omega^8S^{15}$ or is suspension of a spherical class in $\Omega^7S^7$ under stablisation map $\Omega^7S^7\to\Omega^8S^8$. In particular, there are spherical classes arising from Hopf invariant one elements $\eta,\nu,\sigma$ and Kervaire invariant one elements $\eta^2,\nu^2,\sigma^2$ in $H_*\Omega^8S^8$.
\end{crl}

Note that spherical classes in $\Omega^8S^{15}$ arise from $1$ and $\sigma$ which map to $\sigma$ and $\sigma^2$ under $\Omega\sigma:\Omega S^{15}\to\Omega S^8$. By previous theorem spherical classes in $\Omega^8S^{15}$ are known. Hence, the computations in the above theorem will be complete once we know spherical classes in $\Omega^7S^7$.

\begin{rmk}
Let's note that the existence of Hopf fibrations, as well as James-Hopf maps together with Freudenthal suspension theorem all work well integrally. Also, the computations of Herewicz map on decomposable elements maybe carried out integrally as done in Theorem \ref{main0}. We then conclude that the above results may be stated integrally, possibly with some modifications due to omitting the odd torsion.
\end{rmk}

\begin{crl}
(i) The only spherical class in $H_*\Omega S^2$ corresponds to $\eta\in\pi_1^s\simeq\pi_2QS^1$. In fact there is such a class given by $S^2\to\Omega S^2$ provided by the adjoint of $\eta:S^3\to S^2$.\\
(ii) The only spherical class in $H_*\Omega^iS^{3+i}$, $1\leqslant i\leqslant3$, corresponds to $\nu\in\pi_3^s\simeq\pi_6QS^3$. In fact there is such a class give by $S^6\to\Omega S^4\stackrel{E^{i-1}}{\to}\Omega^iS^{4+i}$. Moreover, there is no spherical class in $H_*\Omega^jS^{4+i}$ for $j<i$.\\
(iii) The only spherical class in $H_*\Omega^iS^{7+i}$, $1\leqslant i\leqslant7$, corresponds to $\sigma\in\pi_7^s\simeq\pi_{14}QS^7$. In fact there is such a class given by
$S^{14}\to\Omega S^8\stackrel{E^{i-1}}{\to}\Omega^iS^{8+i}$. Moreover, there is no spherical class in $H_*\Omega^jS^{8+i}$ for $j<i$.
\end{crl}

\begin{proof}
All of these follow from the above theorem together with the fact that the suspension map $\Omega S^{k+1}\to\Omega^{i}S^{k+i}$ induces a monomorphism in homology, hence preserving spherical classes. The nonexistence part, follows similar to the previous theorem by adjointing down. We do (ii) for illustration. Consider the iterated suspension map $\Omega^iS^{3+i}\to\Omega^4S^7$. Hence, for $f:S^n\to\Omega^iS^{3+i}$ with $f_*\neq 0$ the composition $S^n\to\Omega^iS^{4+i}\to\Omega^4S^7$ is nontrivial in homology. Therefore, $n=6$ and $f:S^6\to\Omega^iS^{i+3}$ which maps to $\nu$ by the above theorem. In fact, for $i>1$, $f\in\pi_{6+i}S^{3+i}\simeq\pi_3^s\simeq\Z/8$. Hence, $f=E^i\nu$. For $i=1$, $f:S^6\to\Omega S^4$ is nontrivial, if and only if $f_*=g_3^2$, hence the adjoint of $f$ is detected by the unstable Hopf invariant, which working at the prime $2$ means it is detected by $Sq^4$ in its mapping cone. This means $f$ is an odd multiple of $S^7\to S^4$.
\end{proof}

Finally, we conclude by computing spherical classes in $H_*\Omega^2S^2$.
\begin{thm}
The only spherical classes in $H_*\Omega^2S^2$ arise from the Hopf invariant one and Kervaire invariant one elements $\eta$ and $\eta^2$, respectively.\\
%(ii) The only spherical classes in $H_*\Omega^3S^3$ are precisely the image of spherical classes in $H_*\Omega^2S^2$, namely the Hopf invariant and Kervaire invariant one elements $\eta$ and $\eta^2$ respectively.
\end{thm}

\begin{proof}
We apply the homotopy equivalence $\Omega^2\eta$. By this equivalence, any spherical class in $H_*\Omega^2S^2$ is image of a spherical class in $H_*\Omega^2S^3$. By the above computations, spherical classes in $H_*\Omega^2S^3$ arise from $1,\eta\in\pi_*\Omega^2S^3$ which map to $\eta$ and $\eta^2$ in $\pi_*\Omega^2S^2$ under $(\Omega^2\eta)_\#:\pi_*\Omega^2S^3\to\pi_*\Omega^2S^3$. This completes the proof.\\
%(ii) The suspension map $\Omega^2S^2\to\Omega^3S^3$ induces a monomorphism in homology, hence $\eta$ and $\eta^2$ survive to spherical classes in $H_*\Omega^3S^3$. Next, we use the decomposition $\Omega S^4\simeq S^3\times \Omega S^7$. From this, a spherical class in $\Omega S^3$ gives rise to a spherical class in $\Omega^2S^4$ ???
\end{proof}

\subsection{Further computations}
The aim of this section is to compute spherical classes in more loop spaces associated to spheres. For instance, the decomposition $\Omega S^4\simeq S^3\times\Omega S^7$ provides a decomposition of $2$-loop spaces $\Omega^3S^4\simeq \Omega^2S^3\times\Omega^3S^7$. The spherical classes in the second factor arise from inclusion of the bottom cell. Noting that the decomposition of $\Omega S^7$ off $\Omega S^4$ is provided by $\Omega\nu:\Omega S^7\to\Omega S^4$, the bottom cell of $\Omega S^7$ gives rise to a spherical class in $\Omega^3S^4$ detected by $\nu$. The other spherical classes arise from the first factor, that is the image of $E:\Omega^2S^3\to\Omega^3S^4$. Moreover, note that by adjointing down, we may use a spherical classes in $\Omega^iS^4$, with $i<3$, to get a spherical classes in homology of $\Omega^3S^4$. Since Hurewicz image of $\eta:S^2\to\Omega S^2$ is a square, and since $\Omega E:\Omega S^2\to\Omega^3S^4$ is  loop map, then the Hurewicz image of $\eta:S^2\to\Omega^3S^4$ remains a square which will die under homology suspension $H_*\Omega^3S^4\to H_{*+1}\Omega^2S^4$. We then have proved the following.
\begin{lmm}
The only spherical classes in $H_*\Omega^3S^4$ arise from $\eta,\nu$ and the inclusion of the bottom cell $S^1\to\Omega^3S^4$. Moreover, the only spherical classes in $H_*\Omega^iS^4$ with $1\leqslant i<3$ arise from the inclusion of the bottom cell and $\nu$.
\end{lmm}
A similar reasoning as above, using the decomposition $\Omega S^8\simeq S^7\times\Omega S^{15}$, proves the following.
\begin{lmm}
The only spherical classes in $H_*\Omega^5S^8$ arise from $\sigma$, $\nu$ and the inclusion of the bottom cell $S^3\to\Omega^5S^8$. Moreover, the only spherical classes in $H_*\Omega^iS^4$ with $1\leqslant i<5$ arise from the inclusion of the bottom cell and $\sigma$.
\end{lmm}
Next, note that as the suspension $\Omega^3S^4\to\Omega^7S^8$ induces a monomorphism, hence we get spherical classes given by the inclusion of the bottom cell, $\eta$ and $\nu$. We also obtain, a spherical class given $\sigma$ by adjointing down the spherical class given by $\sigma$ in $\Omega^5S^8$.

%\bibliographystyle{plain}
%\bibliography{spherical}

\bibliographystyle{plain}

\end{document}